\documentclass[11pt]{article}
\usepackage[margin=1in]{geometry}
\usepackage{amsmath,amssymb,amsthm,verbatim}

\title{The Correct Exponent for the Gotsman-Linial Conjecture}
\author{Daniel M. Kane}

\newcommand{\R}{\mathbb{R}}

\newcommand{\var}{\textrm{Var}}

\newcommand{\pr}{\textrm{Pr}}
\newcommand{\sgn}{\textrm{sgn}}
\newcommand{\E}{\mathbb{E}}

\newcommand{\poly}{\textrm{poly}}

\newcommand{\Inf}{\textrm{Inf}}

\newcommand{\as}{\mathbb{AS}}
\newcommand{\ns}{\mathbb{NS}}
\newcommand{\gas}{\mathbb{GAS}}

\newcommand{\tth}{^\textrm{th}}
\newcommand{\mas}{\textrm{MAS}}
\newcommand{\mras}{\textrm{MRAS}}
\newcommand{\masa}{\textrm{MASa}}
\newcommand{\mrasa}{\textrm{MRASa}}

\newtheorem{thm}{Theorem}
\newtheorem{prop}[thm]{Proposition}
\newtheorem{cor}[thm]{Corollary}
\newtheorem{lem}[thm]{Lemma}
\newtheorem{conj}[thm]{Conjecture}

\newtheorem*{defn}{Definition}
\newtheorem*{rmk}{Remark}

\begin{document}
\maketitle

\section{Introduction}

We recall that a (degree-$d$) \emph{polynomial threshold function} (or PTF) is a function of the form $f(x)=\sgn(p(x))$ for some fixed (degree-$d$) polynomial $p$.  Polynomial threshold functions have found application in many areas of computer science, but many fundamental questions about them remain open.  Perhaps one of the longest standing of these problems is that of bounding the sensitivity of such functions.  This question was first considered in detail in \cite{gl} where it was conjectured that:

\begin{conj}[Gotsman-Linial Conjecture]\label{GLConj}
Let $f$ be a degree-$d$ polynomial threshold function in $n>1$ variables, then it's average sensitivity (for the definition of average sensitivity see Section \ref{sensDefnSec}) is bounded by
$$
\as(f) \leq 2^{-n+1} \sum_{k=0}^{d-1} \binom{n}{\lfloor (n-k)/2 \rfloor}(n-\lfloor (n-k)/2 \rfloor).
$$
\end{conj}

It should be noted that if Conjecture \ref{GLConj} holds, then the stated bound would in fact be tight for $f$ defined by the product of the linear polynomials that cut through the middle $d$ layers of the hypercube.  It is also of interest to note the asymptotics of the bound given in Conjecture \ref{GLConj}.  In particular, for $n\gg d^2$ the upper bound given is $\Theta(d\sqrt{n})$.  Furthermore, by results in \cite{sens2} and \cite{DD}, Conjecture \ref{GLConj} would also imply asymptotically tight bounds for several other measures of sensitivity.

In this work, we prove a new bound on the average sensitivity of a polynomial threshold function and show in particular that for fixed degree that the exponent of $n$ given by Conjecture \ref{GLConj} is correct.

\begin{thm}\label{newMainThm}
Let $f$ be a degree-$d$ polynomial threshold function in $n>1$ variables, then
$$
\as(f) \leq \sqrt{n}(\log(n))^{O(d\log(d))}2^{O(d^2\log(d))}.
$$
\end{thm}

Again by reductions from \cite{sens2} and \cite{DD}, this would also imply new bounds on the noise sensitivity and Gaussian average sensitivity of polynomial threshold functions.  Namely,

\begin{cor}
For $f$ a degree-$d$ polynomial threshold function in $n>1$ variables, and for $1/2>\delta>0$, then
$$
\ns_\delta(f) = \sqrt{\delta}(\log(\delta^{-1}))^{O(d\log(d))}2^{O(d^2\log(d))},
$$
and
$$
\gas(f) = \sqrt{n}(\log(n))^{O(d\log(d))}2^{O(d^2\log(d))}.
$$
\end{cor}

\subsection{Previous Work}

Proving the conjectured bounds for the various notions of sensitivity has proved to be quite difficult.  The degree-$1$ case of Conjecture \ref{GLConj} was known to Gotsman and Linial.  The first non-trivial bounds for higher degrees were obtained independently by \cite{sens2} and \cite{sense3}, who later combined their papers into \cite{sensitivity}.  They essentially proved bounds on average sensitivities of $O_d(n^{1-1/O(d)})$ and bounds on noise sensitivities of $O_d(\delta^{1/O(d)})$.  For the special case of Gaussian noise sensitivity, the author proved essentially optimal bounds in \cite{GNSBound} of $O(d\sqrt{\delta})$.  More recently, in \cite{DD}, the author managed to use this result to get an improved estimate for the Bernoulli case giving a bound on average sensitivity of $O_{c,d}(n^{5/6+c})$ for any $c>0$, for the first time obtaining an exponent of $n$ bounded away from $1$ even as $d$ goes to infinity.  In this work, we improve this bound further, yielding the correct exponent.

\subsection{Overview of our Technique}

We begin with a very high level overview of our technique.  A somewhat more detailed overview can be found below in Section \ref{weakBoundSec}.  Very roughly, our bound is obtained via a recursive bound in terms of $n$.  We begin by splitting our coordinates into  $b$ roughly equally sized blocks (for $b=n^{1/\Theta(d)}$).  The average sensitivity is then the sum over blocks of the expected average sensitivity of a random restriction of the function to a block.  Our bound will follow from the claim that on average all but $\tilde O(\sqrt{b})$ of these blocks correspond to polynomials with standard deviations much smaller then their means, and thus have constant sign with high probability.  This result is obtained by considering the relative sizes of $p$ and its derivative at random points.  Using the idea of strong anticoncentration from \cite{GPRG} (see Lemma \ref{stronganticoncentrationLem} below), we know that on Gaussian inputs that $p$ is likely not much smaller than its derivative.  We bring this result into the Bernoulli setting by way of an invariance principle and regularity lemma, completing the proof.

This paper is organized as follows.  In Section \ref{bacgroundSec}, we provide some notation and basic results.  In Section \ref{overviewSec}, we provide a more detailed version of the above, providing a sketch of a proof of the weaker bound $\as(f) \leq \sqrt{n}\exp\left(O(d\log\log(n))^2 \right).$  We then discuss the modifications necessary to obtain our stronger bound, and introduce some additional tools.  Finally in Section \ref{MainSec}, we prove Theorem \ref{newMainThm}.

\section{Background and Notation}\label{bacgroundSec}

\subsection{Notation}

Throughout we will use $X,Y,Z$ to represent standard multidimensional Gaussian random variables and $A,B,C$ to represented standard multidimensional Bernoulli variables unless otherwise specified.  For a function $f:\R^n\rightarrow \R$, and a vector $v\in \R^n$, we let $D_v f(x)$ be the directional derivative of $f$ at $x$ in the direction of $v$, or equivalently, $D_v f(x) = v \cdot \nabla f(x)$.  For completeness, we formally state the definition of a polynomial threshold function:

\begin{defn}
A function $f:\R^n\rightarrow\R$ is a (degree-$d$) polynomial threshold function if it is of the form
$$
f(x) = \sgn(p(x))
$$
for some (degree-$d$) polynomial $p:\R^n\rightarrow\R$.
\end{defn}

\subsection{Polynomials with Random Inputs}

Here we review some of the basic distributional results about polynomials evaluated at random Gaussian or Bernoulli inputs.  To begin with we define the standard $L^t$ norms:

\begin{defn}
If $f:\R^n\rightarrow\R$ is a function and $t\geq 1$ is a real number we let
$$
|f|_t = \left( \E[|f(X)|^t] \right)^{1/t}, \ \ \  |f|_{B,t} = \left( \E[|f(A)|^t] \right)^{1/t}.
$$
Recall that above $X$ is a standard $n$-dimensional Gaussian and $A$ a standard $n$-dimensional Bernoulli random variable.
\end{defn}

The following Lemma relating the $L^2$ norms will prove to be important:
\begin{lem}\label{L2EqLem}
If $p$ is a multilinear polynomial then
$$
|p|_2 = |p|_{B,2}
$$
\end{lem}
\begin{proof}
This follows immediately upon noting that the polynomials of the form $\prod_{i\in S} x_i$ for subsets $S\subseteq\{1,\ldots,n\}$ form an orthonormal basis for the set of multilinear polynomials with respect to both the inner product defined by the Gaussian measure and the inner product defined by the Bernoulli measure.
\end{proof}

One of the most important results on the distribution of the values of polynomials is the hypercontractivity result which relates the values of higher moments to the second moment.  In particular, the following follows from results of \cite{BHyp} and \cite{hypercontractivity} :

\begin{lem}\label{hypLem}
Let $p$ be a polynomial of degree-$d$ and $t\geq 2$ a real number.  Then
$$
|p|_{t} \leq \sqrt{t-1}^d |p|_{2}, \ \ \ |p|_{B,t} \leq \sqrt{t-1}^d |p|_{B,2}.
$$
\end{lem}

These bounds on higher moments allow us to prove concentration bounds on the distribution of our polynomial.  In particular, we have the following corollary (see \cite{DD}):

\begin{cor}\label{concentrationCor}
For $p:\R^n\rightarrow \R$ a degree-$d$ polynomial $N>0$, then
$$
\pr(|p(X)| > N|p|_{2}) = O\left(2^{-(N/2)^{2/d}} \right), \ \ \ \pr(|p(A)| > N|p|_{B,2}) = O\left(2^{-(N/2)^{2/d}} \right).
$$
\end{cor}

In addition to this concentration result, we will also need some anticoncentration results (i.e. results that tell us that the value of $p$ does not lie in a small interval with too large a probability).  For starters, applying the Paley-Zygmund inequality (see \cite{PZ}) to $p^2$, we obtain the following result, which we call ``weak anticoncentration'':

\begin{cor}[Weak Anticoncentration]\label{WeakAnticoncentrationCor}
Let $p$ be a degree-$d$ polynomial in $n$ variables.  Then
$$
\pr\left(|p(X)|\geq |p|_2/2\right) \geq 9^{-d}/2, \ \ \ \pr\left(|p(A)|\geq |p|_{B,2}/2\right) \geq 9^{-d}/2.
$$
\end{cor}

While the bounds in Corollary \ref{WeakAnticoncentrationCor} are fairly weak, not much more can be said in the Bernoulli case.  In particular, it is not hard to demonstrate non-zero, degree-$d$ polynomials $p$ so that $p(A)=0$ with probability $1-2^{-d}$.  On the other hand, in the Gaussian case it can be shown that the output of $p$ is bounded away from zero with large probability.  In particular, we have the following result of Carbery and Wright (\cite{anticoncentration}):

\begin{lem}[Carbery and Wright]\label{anticoncentrationLem}
If $p$ is a degree-$d$ polynomial and $\epsilon>0$ then
$$
\pr(|p(X)| \leq \epsilon|p|_2) = O(d\epsilon^{1/d}).
$$
\end{lem}

Perhaps more importantly though for our purposes the idea of strong anticoncentration, introduced in \cite{GPRG}, which relates the size of a polynomial to its derivative. In particular we will need:

\begin{lem}[Strong Anticoncentration]\label{stronganticoncentrationLem}
Let $p$ be a non-zero degree-$d$ polynomial and $\epsilon>0$, then
$$
\pr(|p(X)| \leq \epsilon |D_Y p(X)|) = O(d^2 \epsilon).
$$
\end{lem}
\begin{proof}
For real number $\theta$ let
\begin{align*}
X_\theta & = \ \ \ \cos(\theta)X + \sin(\theta) Y,\\
Y_\theta & = -\sin(\theta)X + \cos(\theta)Y.
\end{align*}
We note for any $\theta$ that $X_\theta$ and $Y_\theta$ are independent standard Gaussians.  Taking $\theta$ to be uniformly distributed over $[0,2\pi]$, we have that
\begin{align*}
\pr(|p(X)| \leq \epsilon |D_Y p(X)|) & = \pr(|p(X_\theta)| \leq \epsilon |D_{Y_\theta} p(X_{\theta})|)\\
& = \pr\left(|p(X_\theta)| \leq \epsilon \left|\frac{\partial}{\partial \theta} (p(X_{\theta}))\right|\right)\\
& = \E_{X,Y}\left[\pr_\theta\left(|p(X_\theta)| \leq \epsilon \left|\frac{\partial}{\partial \theta} (p(X_{\theta}))\right|\right)\right].
\end{align*}
We claim that for any $X,Y$ that do not leave $p(X_\theta)$ identically $0$ that the inner probability is $O(d^2\epsilon)$.  We may write $p(X_\theta)$ as a degree-$d$ polynomial in $\sin(\theta)$ and $\cos(\theta)$.  Thus we may write $p(X_\theta)$ as $e^{-id\theta}q(e^{i\theta})$ for some polynomial $q$ of degree at most $2d$.  Letting $z=e^{i\theta}$ we have that
$$
\frac{ \left|\frac{\partial}{\partial \theta} (p(X_{\theta}))\right|}{|p(X_\theta)|} = \left|-dz^{-d} + \frac{q'(z)}{q(z)} \right| \leq d + \left|\frac{q'(z)}{q(z)} \right|.
$$
Now if $\epsilon > 1/(2d)$, we have nothing to prove.  Otherwise, it suffices to bound the probability that the logarithmic derivative of $q$ at $z$ has absolute value at most $1/(2\epsilon)$.  We may factor $q$ as $q(z) = c\prod_{i=1}^{g} (z-r_i)$ where $g\leq 2d$ and $c,r_i$ are some complex numbers.  We have that
\begin{align*}
\left|\frac{q'(z)}{q(z)} \right|  & = \left| \sum_{i=1}^g \frac{1}{z-r_i} \right| \\
& \leq \frac{2d}{\min_i |z-r_i|}.
\end{align*}
Hence we have that $|p(X_\theta)| \leq \epsilon \left|\frac{\partial}{\partial \theta} (p(X_{\theta}))\right|$ only if $|z-r_i| < 4d\epsilon$ for some $i$.  By the union bound over $i$, this happens with probability at most $2dO(4d\epsilon) = O(d^2\epsilon).$  This completes our proof.
\end{proof}
\begin{rmk}
A tighter analysis will actually achieve a bound of $O(d\log(d)\epsilon)$, which is optimal.
\end{rmk}

Finally, we will need a single result on the average size of the derivative of a polynomial.  In particular the following follows from results in \cite{DD}:

\begin{lem}\label{derSizeLem}
For $p$ a degree-$d$ polynomial, then
$$
\var(p(X)) \leq \E[|D_Y p(X)|^2] = \E[|\nabla p(X)|^2] \leq d \var(p(X)).
$$
\end{lem}

\subsection{Sensitivity and Influence}\label{sensDefnSec}

We now define the $i\tth$ influence of a function on the hypercube.
\begin{defn}
If $f:\{-1,1\}^n\rightarrow \R$ and $i$ is an integer between $1$ and $n$, we define
$$
\Inf_i(f) = \E_A[\var_{A_i}(f(A))].
$$
This is the average over ways of picking the values of all coordinates except for the $i\tth$ of the variance over the $i\tth$ coordinate of $f$.  Alternatively it is
$$
\frac{1}{4}\E[|f(A)-f(A^i)|^2]
$$
where $A^i$ is obtained from $A$ by negating the $i\tth$ coordinate.  Finally, if $f$ is given as a multilinear polynomial on $\R^n$ it is not hard to show that
$$
\Inf_i(f) = \left| \frac{\partial f}{\partial x_i} \right|_2^2.
$$
\end{defn}

The last definition may be combined with Lemma \ref{derSizeLem} to obtain the following Corollary:

\begin{cor}\label{totalInfCor}
If $p$ is a multilinear, degree-$d$ polynomial in $n$ variables, then
$$
\var(p(A)) \leq \sum_{i=1}^n \Inf_i(p) \leq d\var(p(A)).
$$
\end{cor}

An important notion is that of regularity of a polynomial, which is a measure of how much influence any one coordinate can have on the output.  We recall:

\begin{defn}
We say that a polynomial $p$ is $\tau$-regular for some $\tau>0$ if
$$
\Inf_i(p) \leq \tau \var(p(A))
$$
for all $i$.
\end{defn}

We also recall the definition of the average sensitivity (also known as the total influence) of a Boolean function.

\begin{defn}
If $f:\{-1,1\}^n\rightarrow \{-1,1\}$ then
$$
\as(f) := \sum_{i=1}^n \Inf_i(f).
$$
\end{defn}

Finally, we define some functions to keep track of the maximum possible average sensitivity of a polynomial threshold function of a given dimension, degree, and amount of regularity.

\begin{defn}
If $d,n,\tau>0$ are real numbers we let $\mas(d,n)$ be the maximum over polynomial threshold functions $f$ of degree at most $d$ and dimension at most $n$ of $\as(f)$.  We let $\mras(d,n,\tau)$ be the maximum over such functions $f$ where additionally $f(x)=\sgn(p(x))$ for $p$ a degree-$d$, $\tau$-regular polynomial of $\as(f)$.
\end{defn}

\subsection{Invariance and Regularity}

An important tool for us will be the invariance principle of \cite{MOO}, which relates the distribution of a polynomial under Gaussian input to its distribution under Bernoulli input.  In particular, we have:

\begin{thm}[The Invariance Principle (Mossel, O'{}Donnell, and Oleszkiewicz)]\label{InvarianceThm}
If $p$ is a $\tau$-regular, degree-$d$ multilinear polynomial, and $t\in \R$, then
$$
\left| \pr(p(X) \leq t) - \pr(p(A)\leq t)\right| = O(d\tau^{1/(8d)}).
$$
\end{thm}

We will need a theorem similar to Theorem \ref{InvarianceThm}.  The following is proved by nearly identical means to Theorem \ref{InvarianceThm}:

\begin{prop}\label{absInvProp}
Let $p$ and $q$ be degree-$d$, multilinear polynomials in $n$ variables.  Suppose for some $\tau>0$ that $\Inf_i(p),\Inf_i(q)\leq \tau$ for all $i$.  Suppose furthermore that $|p+q|_2,|p-q|_2\geq 1.$  Then
$$
\pr(|p(A)|\leq |q(A)|) = \pr(|p(X)|\leq |q(X)|) + O(d\tau^{1/(8d)}).
$$
\end{prop}
\begin{proof}
We note that it suffices to prove only that
$$
\pr(|p(A)|\leq |q(A)|) \leq \pr(|p(X)|\leq |q(X)|) + O(d\tau^{1/(8d)})
$$
and to note that the other direction follows from interchanging $p$ and $q$.

Note that
$$
\pr(|p(A)|\leq |q(A)|) = \pr(|p(A)|\leq q(A))  + \pr(|p(A)|\leq -q(A)),
$$
and
$$
\pr(|p(X)|\leq |q(X)|) = \pr(|p(X)|\leq q(X))  + \pr(|p(X)|\leq -q(X)),
$$
it suffices to show that
$$
\pr(|p(A)|\leq q(A)) \leq \pr(|p(X)|\leq q(X)) + O(d\tau^{1/(8d)}).
$$

Letting $r=q-p$ and $s=q+p$, we need to show that
\begin{equation}\label{reformulatedAbsValEqn}
\pr(r(A)\geq 0 \textrm{ and } s(A)\geq 0) \leq \pr(r(X)\geq 0 \textrm{ and } s(X)\geq 0) + O(d\tau^{1/(8d)}),
\end{equation}
where $r$ and $s$ are polynomials of degree-$d$, $L^2$ norm at least $1$, and maximum influence at most $\tau$.  By rescaling $r$ and $s$, we may assume that $|r|_2=|s|_2=1$.

Let $\rho$ be a smooth function so that $\rho(x)=1$ for $x>0$, $\rho(x)=0$ for $x< - \tau^{1/8}$, and $0\leq \rho(x)\leq 1$ for all $x$.  We note that such $\rho$ can be found with $|\rho^{(k)}(x)| = O(\tau^{-k/8})$ for all $x$ and all $1\leq k \leq 3$.  Define
$$
g(x) := \psi(r(x),s(x)) := \rho(r(x))\rho(s(x)).
$$
Since $g(x)=1$ whenever $r$ and $s$ are both positive,
$$
\pr(r(A)\geq 0 \textrm{ and } s(A)\geq 0) \leq \E[g(A)].
$$

We claim that
$$
|\E[g(A)]-\E[g(X)]| \leq 2^{O(d)}\tau^{1/8}.
$$
This follows immediately from Theorem 4.1 of \cite{MInv}, noting that $B=O(\tau^{-3/8})$.  Notice that if $\tau > d^{-d}$ that we have nothing to prove and that otherwise $2^{O(d)}\tau^{1/8}= O(d\tau^{1/(8d)})$.

We now need to bound the expectation of $g(X)$.  We note that $g(X)$ is 0 unless $r(X),s(X)\geq -\tau^{1/8}$.  This can happen only if either both are positive or at least one has absolute value at most $\tau^{1/8}$.  Thus
$$
\E[g(X)] \leq \pr(r(X)\geq 0 \textrm{ and } s(X)\geq 0) + \pr(|r(X)| \leq \tau^{1/8}) + \pr(|s(X)| \leq \tau^{1/8}).
$$
By Lemma \ref{anticoncentrationLem}, this is at most
$$
\pr(r(X)\geq 0 \textrm{ and } s(X)\geq 0) + O(d\tau^{1/(8d)}).
$$
Thus,
\begin{align*}
\pr(r(A)\geq 0 \textrm{ and } s(A)\geq 0) & \leq \E[g(A)] \\ & \leq \E[g(X)] + O(d\tau^{1/(8d)}) \\ & \leq \pr(r(X)\geq 0 \textrm{ and } s(X)\geq 0) +O(d\tau^{1/(8d)}).
\end{align*}
This proves Equation \eqref{reformulatedAbsValEqn}, and completes our proof.
\end{proof}

The invariance principle will turn out to be very useful to apply to regular polynomials, but for general polynomials we will need a way to reduce to this case.  For this purpose we can make use of the following result of \cite{reg}:

\begin{thm}[Diakonikolas, Servedio, Tan, Wan]\label{regLemThrm}
Let $f(x) = sign(p(x))$ be any degree-$d$ PTF. Fix any $\tau > 0$. Then $f$ is equivalent to a decision
tree $T$ , of depth
$$
\textrm{depth}(d,\tau) = \frac{1}{\tau}\cdot (d \log (\tau^{-1}))^{O(d)}
$$
with variables at the internal nodes and a degree-$d$ PTF $f_\rho = \sgn(p_\rho)$ at each leaf $\rho$, with the following
property: with probability at least $1 -\tau$, a random path from the root reaches a leaf $\rho$ such that $f_\rho$ is
$\tau$-close to some $\tau$-regular degree-$d$ PTF.
\end{thm}

Unfortunately, for our purposes, we will also require a stronger version of this Theorem.

\begin{prop}\label{regLem2Prop}
Let $p$ be a degree-$d$ polynomial on the hypercube and let $1/4> \tau,\epsilon,\delta >0$ be real numbers.  Then $p$ can be written as a decision tree of depth at most
$$
D = \tau^{-1}\left(d\log(\tau^{-1})\log(\epsilon^{-1})\right)^{O(d)}\log(\delta^{-1})
$$
with variables at the internal nodes and a degree-$d$ polynomial threshold function $f_\rho = \sgn(p_\rho)$ at each leaf $\rho$, with the following
property: that for a random leaf, $\rho$, with probability $1-\delta$ we have that $p_\rho$ is either $\tau$-regular, or constant sign with probability at least $1-\epsilon$.
\end{prop}

Proposition \ref{regLem2Prop} will follow from repeated application of the following Lemma.
\begin{lem}\label{oneStepRegLem}
Let $p$ be a degree-$d$ polynomial on the hypercube and let $1/4 > \tau,\epsilon>0$ be real numbers.  There exists a set $S$ of coordinates with
$$
|S| \leq \tau^{-1}\left(d\log(\tau^{-1})\log(\epsilon^{-1})\right)^{O(d)}
$$
so that after assigning random values to the coordinates of $S$, with probability at least $2^{-O(d)}$ over the choice of assignments, the restricted polynomial $p_\rho$ is either $\tau$-regular or has constant sign with probability at least $1-\epsilon$.
\end{lem}
\begin{proof}
We assume without loss of generality that $p$ is multilinear with $|p|_2=1$.  We take $S$ to simply be the set of all coordinates of influence more than
$$
\tau ( d \log(\tau^{-1})\log(\epsilon^{-1}))^{-Md}
$$
for $M$ a sufficiently large constant.  We have that $|S|$ will be of the appropriate order since the total influence of $p$ is at most $d$.  We claim that with probability at least $2^{-O(d)}$ that both of the following hold:
\begin{equation}\label{smallResEqn}
|p_\rho|_2^2 \geq 1/2.
\end{equation}
\begin{equation}\label{maxInfEqn}
\max_i(\Inf_i(p_\rho)) \leq \tau (4 \log(\epsilon^{-1}))^{-d/2}.
\end{equation}

As for Equation \eqref{smallResEqn}, we note that $|p_\rho|_2^2$ is a polynomial of degree at most $2d$ in the assignments of the coordinates in $S$.  Furthermore its expectation is $|p|_2^2$.  Therefore, the $L^2$ norm of this polynomial is at least $|p|_2^2=1$, and hence by Corollary \ref{WeakAnticoncentrationCor}, Equation \eqref{smallResEqn} holds with probability $2^{-O(d)}$.  We now need to show that Equation \eqref{maxInfEqn} fails to hold with at most half of this probability.  We note that for each $i$ that $\Inf_i(p_\rho)$ is the sum of squares of degree-$d$ polynomials in the assignments of coordinates of $S$, and has mean value $\Inf_i(p)$.  Thus it is given by some degree-$2d$ polynomial, $q$ with $|q|_1=\Inf_i(p)$.  By Corollary \ref{WeakAnticoncentrationCor}, $|q_i|_1 \geq 2^{-O(d)}|q_i|_2/2$, and thus $|q_i|_2 = 2^{O(d)}\Inf_i(p)$.  Now, for each $i\not\in S$, $\Inf_i(p)\leq \tau (d \log(\tau^{-1})\log(\epsilon^{-1}))^{-Md}:=m.$  By Corollary \ref{concentrationCor}, we have that for $M$ sufficiently large
\begin{align*}
\pr\left(\Inf_i(p_\rho) > \tau (4 \log(\epsilon^{-1}))^{-d/2}\right) & \leq m2^{-Md}\exp\left(-d\left(\frac{m}{\Inf_i(p)} \right)^{1/d}\right).
\end{align*}
Since there are at most $d 2^k m^{-1}$ coordinates $i$ for which $\Inf_i(p) \in [m 2^{-k}, m2^{-k+1}]$, the probability that any coordinate of $p_\rho$ has too large an influence is at most
\begin{align*}
\sum_{k=1}^\infty d 2^k m^{-1} m2^{-Md}\exp\left(-d 2^{(k-1)/d}\right) \leq d2^{-Md} 2^{O(d)}\sum_{\ell=0}^\infty 2^{d\ell} \exp\left(-d 2^{\ell}\right)& \leq 2^{O(d)}d2^{-Md},
\end{align*}
which is sufficiently small.

Now if Equations \ref{smallResEqn} and \ref{maxInfEqn} both hold, then either $\var(p_\rho)\geq (4\log(\epsilon^{-1}))^{-d/2}$, in which case $p_\rho$ is $\tau$-regular, or $\var(p_\rho) \leq (4\log(\epsilon^{-1}))^{-d/2}.$  In the latter case, since $1/2 \leq |p_\rho|_2^2 = \var(p_\rho) + \E[p_\rho]^2$, we have that letting $\mu = \E[p_\rho]$ that $|\mu|\geq 1/2$.  Furthermore, $|p_\rho -\mu|_2^2 = \var(p_\rho) \leq (4\log(\epsilon^{-1}))^{-d/2}.$  Therefore, by Corollary \ref{concentrationCor}, we have with probability at least $1-\epsilon$ that
$$
|p_\rho(A)-\mu| < |\mu|.
$$
And thus with probability at least $1-\epsilon$, $p_\rho$ has the same sign as $\mu$.  This completes our proof.
\end{proof}

Proposition \ref{regLem2Prop} now follows from applying the construction in Lemma \ref{oneStepRegLem} repeatedly to the leaves that do not yet satisfy one of the necessary conditions up to a total of at most $2^{O(d)}\log(\delta^{-1})$ times.

\section{Overview of our Technique}\label{overviewSec}

\subsection{Proof of a Simpler Bound}\label{weakBoundSec}

We begin by providing a somewhat detailed sketch of a proof of the slightly weaker bound that
$$
\mas(d,n) \leq \sqrt{n}\exp\left( O(d \log\log(n))^2 \right).
$$
Starting with a polynomial threshold function $f=\sgn(p(x))$ for $p$ a degree-$d$ multilinear polynomial threshold function in $n$ variables, we begin by using Theorem \ref{regLemThrm} to reduce to the case where $p$ is $n^{-1/2}$-regular, introducing an error of $\sqrt{n}O(d\log(n))^{O(d)}$ in the process.  We then split the coordinates into $b$ blocks of roughly equal size for $b= n^{1/\Theta(d)}$, and note that the sensitivity of $f$ is the sum over blocks of the sensitivity of $f$ randomly restricted to a function on only that block of coordinates.  We note that by Corollary \ref{concentrationCor} that if any of these restrictions have an expected value that exceeds their standard deviation by a factor of more than about $\log(n)^{d/2}$, then the polynomial will have constant sign with high probability and can thus be ignored.  We call a block for which this does not happen \emph{good}.

We thus have that the average sensitivity of $f$ is bounded by the expected number of good blocks times $\mas(d,n/b)$.  It is not hard to show that a polynomial $q$ with standard deviation at least $\log(n)^{-d/2}$ times the absolute value of its expectation, has a reasonable probability of having
$$
\frac{|\nabla q(A)|}{|q(A)|} > 2^{-O(d)}\log(n)^{-d/2}.
$$
This allows one to bound the expected number of good blocks in terms of the expectation of
$$
\max\left(b,\left( \frac{|\nabla p(A)|}{|p(A)|}\right)^2\right).
$$
Or more tractably, in terms of the expectation of
$$
\max\left(b,\left( \frac{|D_B p(A)|}{|p(A)|}\right)^2\right).
$$
On the other hand, we can use Lemma \ref{stronganticoncentrationLem} and Proposition \ref{absInvProp} to show that
$$
\pr\left( \frac{|D_B p(A)|}{|p(A)|} > \sqrt{k} \right) \approx k^{-1/2}
$$
for each $k$.  This lets us bound the expected number of good blocks by $O(\log(n))^{d}\sqrt{b}$.  This provides us with a recursive bound for the average sensitivity, which comes out to roughly
$$
\mas(d,n) \leq O(\log(n))^d n^{1/(16d)} \mas(d,n^{1-1/(8d)}),
$$
which gives the bound required.

Unfortunately, in the above argument, the requirement that we only consider whether or not a block is good has cost us a factor of $\log(n)^d$ at each recursive step, yielding a bound off by a factor of $\exp(d^2\log\log(n)^2)$.  By being less strict with our reductions, we can instead lose only a $\poly(d)$ factor at each step, yielding a bound with only polylogarithmic error.  In order to do this, instead of simply considering whether or not a block is good, we consider more detailed information about the ratio of its value and its derivative at a random point.  To do this we will need to introduce some new machinery, which we do in the next Section.

\subsection{The $\alpha$ Function}

The following will prove to be a key concept for our analysis:

\begin{defn}
For $p$ a non-zero polynomial we let $\alpha(p)$ be defined by
$$
\alpha(p) := \E\left[\min\left( 1 , \frac{|D_B p(A)|^2}{|p(A)|^2}\right)\right].
$$
Similarly, let
$$
\beta(p) := \E\left[\min\left( 1 , \frac{|D_Y p(X)|^2}{|p(X)|^2}\right)\right].
$$
\end{defn}

We will bound the noise sensitivity of a polynomial threshold function in terms of $\alpha(p)$.  First we introduce some notation:
\begin{defn}
Let $\masa(d,n,a)$ be the maximum average sensitivity of a polynomial threshold function $f(x)=\sgn(p(x))$ where $p$ is a polynomial of degree at most $d$ in at most $n$ variables with $\alpha(p)\leq a$.

Let $\mrasa(d,n,a,\tau)$ be the maximum average sensitivity of a polynomial threshold function $f(x)=\sgn(p(x))$ where $p$ is a $\tau$-regular polynomial of degree at most $d$ in at most $n$ variables with $\alpha(p)\leq a$.
\end{defn}

In particular, we will prove:

\begin{prop}\label{newMainProp}
$$
\masa(d,n,a) \leq a\sqrt{n}(\log(n))^{O(d\log(d))}2^{O(d^2\log(d))}.
$$
\end{prop}
Theorem \ref{newMainThm} will follow as an immediate Corollary of Proposition \ref{newMainProp}.

We will require a version of Lemma \ref{stronganticoncentrationLem} that takes $\beta(p)$ into account.  In particular, we use the following:
\begin{lem}\label{alphaStrongAnticoncentrationLem}
Let $p$ be a degree-$d$ polynomial in any number of variables and $1>\epsilon>0$ a real number.  Then
$$
\pr(|p(X)| \leq \epsilon |D_Y p(X)|) = O(d^3\beta(p) \epsilon).
$$
\end{lem}
\begin{proof}
If $\epsilon \geq d^{-3}$, the result follows from the fact that
$$
\pr(|p(X)| \leq \epsilon |D_Y p(X)|) \leq \pr(|p(X)| \leq |D_Y p(X)|) \leq \beta(p).
$$
Thus we may assume that $\epsilon \leq d^{-3}$.

For random $X$,$Y$, let $g(\theta) = p(\cos(\theta)X+\sin(\theta)Y).$  By the proof of Lemma \ref{stronganticoncentrationLem}, we have that the probability in question is $$
\pr(|g(\theta)| \leq \epsilon |g'(\theta)|) \ \ \ \leq \ \ \ O(d^2 \epsilon)\pr_{X,Y}\left(\exists \theta: \frac{|g'(\theta)|}{|g(\theta)|} > \epsilon^{-1} \right).
$$
We note that $g(\theta) = \sum_{m=-d}^d a_m e^{im\theta}$ for some constants $a_m$.  If it is the case that $|a_0| > 2\sum_{m\neq 0}|a_m|$, then $|g(\theta)|\geq |a_0|/2$ for all $\theta$, and $$|g'(\theta)| = \left|\sum_{m\neq 0} |m| |a_m| \right| \leq d |a_0|/2.$$  Thus, in this case, $|g'(\theta)|/|g(\theta)| \leq d < \epsilon^{-1}$ for all $\theta$.  Thus the probability in question is at most
$$
O(d^2 \epsilon)\pr_{X,Y}\left(|a_0| \leq 2 \sum_{m\neq 0} |a_m|\right).
$$
When $|a_0| \leq 2 \sum_{m\neq 0} |a_m|$, we have $|g(\theta)|\leq \sum |a_m|$ for all $\theta$, and the average value of $|g'(\theta)|^2$ is $$\sum_{m \neq 0} m^2 |a_m|^2 \geq \left(\sum |a_m| \right)^2/(8d).$$ Thus, $|g'(\theta)|/|g(\theta)| \geq 1/(8d)$ with constant probability.  Hence we have that
\begin{align*}
\pr(|p(X)| \leq \epsilon |D_Y p(X)|) & = O(d^2 \epsilon)\pr(|p(X)| \leq 8d|D_Y p(X)|) \\
& = O(d^2 \epsilon) \pr(|g(\theta)| \leq 8d|g'(\theta)|) \\
& = O(d^2 \epsilon) \pr(|p(X)| \leq 8d|D_Y p(X)|) \\
& = O(d^3\beta(p) \epsilon).
\end{align*}
\end{proof}

\section{Proof of the Main Theorem}\label{MainSec}

In this Section, we prove Proposition \ref{newMainProp} and thus Theorem \ref{newMainThm}.  We begin in Section \ref{regSec} by proving a recursive bound on average sensitivity for regular polynomial threshold functions.  In Section \ref{redSec}, we show a reduction to the regular case.  Finally, in Section \ref{pfSec}, we combine these recursive bounds to obtain a proof of Proposition \ref{newMainProp}.

\subsection{The Regular Case}\label{regSec}

Here we prove the reduction in the case of a regular polynomial.  In particular, we show:

\begin{prop}\label{NewBlockProp}
Let $d,n,\tau,a>0$ be real numbers and let $b\leq n$ be a positive integer.  Then
$$
\mrasa(d,n,a,\tau) \leq b\E_{\aleph}[\masa(d,n/b+1,\aleph)]
$$
for some non-negative random variable $\aleph$ with $\E[\aleph]=O(d^3 a b^{-1/2} + d^4 \tau^{1/(8d)}).$
\end{prop}
\begin{proof}
Consider $f=\sgn(p(x))$ for $p$ a $\tau$-regular, degree-$d$, multilinear polynomial in at most $n$ dimensions with $\var(p(A))=1$ and $\alpha(p)\leq a$.  It suffices to show that for all such $f$ that
$$
\as(f)  \leq b\E_{\aleph}[\masa(d,n/b+1,\aleph)]
$$
for an appropriate $\aleph$.

We begin by partitioning the coordinates of $f$ into $b$ blocks each of size at most $n/b+1$.  For each block, $\ell$, and Bernoulli random variable, $A$, we let $A^\ell$ be the coordinates of $A$ that do not lie in $\ell$.  We let $p_{A^\ell}$ be the function defined on the coordinates of $\ell$ obtained by plugging these values into $p$ for the other coordinates.  We define $f_{A^\ell}$ similarly.  It is not hard to see that
$$
\as(f) = \sum_\ell \E_{A^\ell}[\as(f_{A^\ell})].
$$

It thus suffices to show that
\begin{equation}\label{expectedAlphaEquation}
\sum_\ell \E_{A^\ell} [\alpha(p_{A^{\ell}})] = O(d^3 a \sqrt{b} + d^4 b \tau^{1/(8d)}).
\end{equation}

We have that
\begin{align*}
\sum_\ell \E_{A^\ell}[\alpha(p_{A^\ell})] & = \sum_\ell \E\left[\min\left(1,\frac{|D_B p_{A^\ell}(A)|^2}{|p(A)|^2}\right) \right]\\
& \leq O\left(\sum_\ell\E\left[\min\left(1,\frac{|\nabla p_{A^\ell}(A)|^2}{|p(A)|^2}\right) \right]\right)\\
& \leq  O\left( \E\left[\min\left(b,\frac{|\nabla p(A)|^2}{|p(A)|^2}\right) \right]\right)\\
& \leq O\left(\E\left[\min\left(b,\left(\frac{|D_B p(A)|}{|p(A)|}\right)^2\right) \right]\right).
\end{align*}

We note that
\begin{align*}
\E  \left[\min\left(b,\left(\frac{|D_B p(A)|}{|p(A)|}\right)^2\right) \right]
& = \E\left[\min\left(1,\left(\frac{|D_B p(A)|}{|p(A)|}\right)^2\right) \right] + \int_{1}^b \pr\left( |p(A)| \leq t^{-1/2}|D_B p(A)| \right) dt\\
& = \alpha(p) + \int_{1}^b \pr\left( |p(A)| \leq t^{-1/2}|D_B p(A)| \right) dt.
\end{align*}

To bound the second term above, we use an invariance principle to relate the necessary probabilities to those in the Gaussian case.  In order to do so we define the polynomial $q(A,B) = D_B p(A)$.  To show that $q$ has small influences we note that
\begin{align*}
\left| \frac{\partial q}{\partial b_i} \right|^2_2 = \left| \frac{\partial p}{\partial a_i} \right|^2_2 = \Inf_i(p)\leq \tau,
\end{align*}
and
\begin{align*}
\left| \frac{\partial p}{\partial a_i} \right|_2^2 & = \E\left[\left|\frac{\partial D_B p(A)}{\partial a_i} \right|^2 \right]\\
& = \E\left[\left|D_B \frac{\partial p(A)}{\partial a_i} \right|^2 \right]\\
& \leq d \E\left[\left|\frac{\partial p(A)}{\partial a_i} \right|^2 \right]\\
& = d \Inf_i(p)\\
& \leq d \tau.
\end{align*}
Where the middle line above is by Lemma \ref{derSizeLem}.  Thus all of the influences of $p$ and $q$ are at most $d\tau$.  Furthermore, it is easy to see that $p$ and $q$ have covariance 0 (since $p$ is even in $B$ and $q$ is odd in terms of $B$).  Thus we have for any real $s$ that
$$
|p+s q|_2 \geq |p|_2 \geq \var(p(A)) = 1.
$$
Therefore by Proposition \ref{absInvProp}, for any real $s$ we have that
\begin{align}\label{SAInvEqn}
\pr(|p(A)| \leq s |D_B p(A)|) & = \pr(|p(X)| \leq s |D_Y p(X)|) +O(d\tau^{1/(8d)}).
\end{align}

Applying Equation \eqref{SAInvEqn}, we find that
\begin{align*}
\beta(p) & = \int_0^1 \pr(|p(X)| \leq s^{-1/2} |D_Y p(X)|) ds\\
& = \int_0^1 \pr(|p(X)| \leq s^{-1/2} |D_Y p(X)|) ds + O(d\tau^{1/(8d)})\\
& = \alpha(p) + O(d\tau^{1/(8d)}).
\end{align*}

By Equation \eqref{SAInvEqn} and Lemma \ref{alphaStrongAnticoncentrationLem}, we find that
\begin{align*}
\int_{1}^b \pr\left( |p(A)| \leq t^{-1/2}|D_B p(A)| \right) dt & = \int_{1}^b \pr\left( |p(X)| \leq t^{-1/2}|D_Y p(X)| \right) dt + O(db\tau^{1/(8d)})\\
& = \int_{1}^b O(d^3 t^{-1/2} \beta(p)) db +O(db\tau^{1/(8d)})\\
& = O(d^3\sqrt{b} \beta(p)) +O( db\tau^{1/(8d)})\\
& = O(d^3 \sqrt{b} \alpha(p) + d^4 b \tau^{1/(8d)}).
\end{align*}

This completes the proof of Equation \eqref{expectedAlphaEquation}, as desired.
\end{proof}

\subsection{Reducing to the Regular Case}\label{redSec}

In this Section, we show by a simple application of Proposition \ref{regLem2Prop} that the average sensitivity of an arbitrary polynomial threshold function can be bounded in terms of the sensitivity of a regular one.  In particular, we show that:

\begin{prop}\label{NewRegReductionProp}
For any $d,n,a,\tau,\epsilon>0$ we have that
$$
\masa(d,n,a) \leq \tau^{-1}(d\log(\tau^{-1})\log(\epsilon^{-1}))^{O(d)} + 3n\epsilon + \E_\aleph[\mrasa(d,n,\aleph,\tau)],
$$
for some non-negative random variable $\aleph$ with $\E[\aleph]=a.$
\end{prop}
\begin{proof}
Let $p$ be a degree-$d$ polynomial in $n$ variables with $\alpha(p)\leq a$.  Let $f=\sgn\circ p$.  We will show that for an appropriately chosen $\aleph$ that
$$
\as(f) \leq \tau^{-1}(d\log(\tau^{-1})\log(\epsilon^{-1}))^{O(d)} + 3n\epsilon + \E_\aleph[\mrasa(d,n,\aleph,\tau)].
$$

We begin by writing $f$ as a decision tree as given to us in Proposition \ref{regLem2Prop} with $\delta$ set to $\epsilon$.  We claim that the average sensitivity of $f$ is at most the depth of the decision tree plus the expectation over leaves of the tree of the average sensitivity of the resulting function.  To show this we note that the average sensitivity of $f$ is equal to the expected number of coordinates, $i$ so that $f(A)$ disagrees with $f(A^i)$, where $A^i$ is obtained from $A$ by flipping the $i\tth$ coordinate.  We compute this probability by first conditioning on the path through the decision tree defined by $A$.  Except for a number of coordinates that is at most the depth of the tree, flipping the $i\tth$ coordinate leaves us in the same leaf.  The expected number of such coordinates that we can flip to change the sign of $f$ is at most the average sensitivity of the function corresponding to that leaf.  The expected number of other coordinates is at most the depth of the decision tree.  This completes the proof of this claim.

Thus we have
$$
\as(f) \leq \tau^{-1}(d\log(\tau^{-1})\log(\epsilon^{-1}))^{O(d)} + \E_{\textrm{leaves }\rho}[\as(f_\rho)].
$$
With probability $1-\epsilon$, $f_\rho$ is either $\tau$-regular or constant sign with probability $1-\epsilon$.  The contribution from the remaining $\epsilon$ probability set of leaves is at most $n\epsilon$, and the contribution from the leaves with nearly constant sign is at most $2n\epsilon$.  We thus need to bound the contribution from the leaves for which $f_\rho$ is $\tau$-regular.  This is an expectation of the average sensitivities of the threshold functions of $\tau$-regular, degree-$d$ polynomials in at most $n$ variables.  We have only to show that
$$
\E[\alpha(p_\rho)] \leq a.
$$
But this follows immediately from the definition of $\alpha$.
\end{proof}

\subsection{Putting it Together}\label{pfSec}

Here we combine Propositions \ref{NewBlockProp} and \ref{NewRegReductionProp} to prove Proposition \ref{newMainProp}.

First we need a Lemma:
\begin{lem}\label{smallAlphaLem}
Let $p$ be a degree-$d$ multilinear polynomial in $n$ variables and let $f(x)=\sgn(p(x))$.  There is a constant $K$, so that
if $\alpha(p) < (K\log(n))^{-d}$ then
$$
\as(f) = O(\alpha).
$$
\end{lem}
\begin{proof}
We note that $\as(f)$ is at most $O(n)$ times the probability that $f$ takes on its less common value.  We note that by Corollary \ref{WeakAnticoncentrationCor} that with probability at least $2^{-O(d)}$ that
$$
|D_B p(A)|^2 \geq \E\left[|D_B p(A)|^2\right]/4 \geq \var(p)/4.
$$
Therefore, by the Markov inequality, there is a probability of at least $2^{-O(d)}$ that this occurs and that additionally
$$
|p(A)| \leq 2^{O(d)} |p|_2.
$$
Hence it is the case that
$$
\frac{\var(p)}{|p|_2^2} \leq 2^{O(d)} \alpha(p).
$$

Let $\mu=\E[p(A)].$  We have that
$$|p-\mu|_2 = \sqrt{\var(p)} \leq 2^{O(d)}|p|_2 \sqrt{\alpha(p)}.$$
Since $|p|_2^2 = \mu^2 + \var(p)$, we also have that
$$|p-\mu|_2 = \sqrt{\var(p)} \leq 2^{O(d)}|\mu| \sqrt{\alpha(p)}.$$
Hence for $\alpha(p) < (K\log(n))^{-d}$ for $K$ sufficiently small, we have by Corollary \ref{concentrationCor} $p(A)$ has the same sign as $\mu$ with probability $1-O(\alpha(p)n^{-1})$, yielding our desired bound.
\end{proof}

\begin{proof}[Proof of Proposition \ref{newMainProp}]
Let $\tau=n^{-1/3}$, $\epsilon=n^{-1}$, and $n>2^{Md^2\log(d)}$ for $M$ a sufficiently large constant.  By Proposition \ref{NewRegReductionProp} we have that
$$
\masa(d,n,a) \leq O(n^{1/2}) + \E_\aleph[\mrasa(d,n,\aleph,n^{-1/3})],
$$
for some $\aleph$ with $\E[\aleph]=O(a)$.  Let $b = \left\lceil n^{1/(16d)} \right\rceil $.  Applying Proposition \ref{NewBlockProp} to the above, we have that
\begin{align}\label{simpleReductionEqn}
\masa(d,n,a)& \leq O(n^{1/2}) + b\E_\aleph[\masa(d,n/b+1,\aleph)] \notag \\ & \leq O(n^{1/2}) + 2n^{1/(16d)}\E_\aleph[\masa(d,n^{1-1/(16d)},\aleph)].
\end{align}
Where above $\E[\aleph] = O(d^3 a b^{-1/2} + d^4 n^{-1/(24d)}).$  Notice that either $a\geq (K\log(n))^{-d}$, in which case, $\E[\aleph] = O(d^3 a b^{-1/2})$, or $a< (K\log(n))^{-d}$, in which case $\masa(d,n,a)<n^{1/3}$ by Lemma \ref{smallAlphaLem}.  Thus in any case, Equation \eqref{simpleReductionEqn} holds for some $\aleph$ with $\E[\aleph] = O(d^3 a b^{-1/2})$.

We now proceed by induction on $n$.  In particular, for a sufficiently large constant $M$, we prove by induction on $n$ that
\begin{equation}\label{finalInductionEquation}
\masa(d,n,a) \leq a \sqrt{n} (\log(n))^{Md\log(d)} 2^{Md^2\log(d)}.
\end{equation}
We begin by showing this for $n<2^{M d^2\log(d)}$.  For such $n$, the bound follows from Lemma \ref{smallAlphaLem} and the trivial bound of $n$.

Next suppose that Equation \eqref{finalInductionEquation} holds for all smaller values of $n$.  Bounding the $\masa(d,n,\aleph)$ terms in Equation \eqref{simpleReductionEqn} recursively, we obtain
\begin{align*}
\masa(d,n,a) & \leq O(n^{1/2}) + O(d^3) n^{1/(16d)} \E[\aleph] n^{1/2 - 1/(32d)} (\log(n)(1-1/(16 d))^{Md\log(d)} 2^{Md^2\log(d)}\\
& = O(n^{1/2}) + a O(d^3) \sqrt{n} (\log(n))^{Md\log(d)} d^{-\Omega(M)}2^{Md^2\log(d)}\\
& = O(n^{1/2}) + a \sqrt{n} (\log(n))^{Md\log(d)} O(d^{3-\Omega(M)})2^{Md^2\log(d)}\\
& \leq O(n^{1/2}) + a \sqrt{n} (\log(n))^{Md\log(d)}2^{Md^2\log(d)}/2.
\end{align*}
Where the last line holds when $M$ is sufficiently large.  Now, if $a> (K\log(n))^{-d}$, then this is at most $a\sqrt{n}(\log(n))^{Md\log(d)}2^{Md^2\log(d)}$, as desired.  If on the other hand, $a\leq (K\log(n))^{-d}$, the same bound follows instead from Lemma \ref{smallAlphaLem}.  In either case we have
$$
\masa(d,n,a) \leq a \sqrt{n} \log(n)^{Md\log(d)}2^{Md^2\log(d)}.
$$
This completes our inductive step and finishes the proof.
\end{proof}

\section{Concluding Remarks}

We believe that using techniques from \cite{DD}, that the bound on average sensitivity can be improved to
$$
\sqrt{n}O_d(\log(n))^{O(\log(d))}.
$$
The basic idea would be to use the diffuse regularity lemma and invariance principle instead of the standard ones in the proof above.  This allows us to take a number of blocks, $b$ polynomial in $n$ rather than $n^{1/\Theta(d)}$.  This decreases the number of rounds in our recursion by a factor of $d$, and thus lowers the asymptotic exponent by a corresponding factor.  Unfortunately, the poor dependence on degree in the technology from \cite{DD}, means that this bound will have perhaps a very bad dependence on $d$.

Although it seems that for fixed $d$ we have obtained nearly the correct asymptotic in terms of $n$, our dependence on $d$ is still fairly bad.  In particular, the bound given in Theorem \ref{newMainThm} does not improve upon the trivial bound of $n$ until $\log(n) \gg d^2\log(d)$.  The reason for this is that in our inductive step, we wish to replace $a b^{-1/2} + d \tau^{1/(8d)}$ by $O(a b^{-1/2})$, so long as $a\geq (K\log(n))^{-d}$ (since otherwise we can use simpler bounds).  On the other hand, using the easily established bound $\masa(d,n,a) = O(na)$, it is not hard to prove the bound $\as(f) \leq O(d^4)n^{1-1/(24d)}$, which is non-trivial for $n=2^{O(d\log(d))}$.

\section*{Acknowledgements}

This work was done with the support of an NSF postdoctoral fellowship.

\end{document}